\documentclass[a4paper,fleqn]{article}
\usepackage{amsmath}
\usepackage{latexsym}
\usepackage{amssymb}
\usepackage{amsfonts}
\usepackage{rsfs}
\usepackage[dvips]{graphics}
\usepackage[all]{xy}
\usepackage{amsthm}
\usepackage{enumerate}
\usepackage{url}

\DeclareMathOperator{\dom}{\mathrm{dom}}
\DeclareMathOperator{\rng}{\mathrm{rng}}

\DeclareMathOperator{\Fin}{\mathrm{Fin}}

\newcommand{\mon}[2]{\mathrm{mon}_{\mathscr{#1}}\left(#2\right)}
\newcommand{\fig}[2]{\mathrm{fig}_{\mathscr{#1}}\left(#2\right)}

\newcommand{\cl}[2]{\mathrm{cl}_{\mathscr{#1}}\left(#2\right)}
\newcommand{\intrr}[2]{\mathrm{int}_{\mathscr{#1}}\left({#2}\right)}

\newcommand{\Cntd}[2]{\mathrm{Cntd}_{\mathscr{#1}}\left(#2\right)}
\newcommand{\Sep}[3]{\mathrm{Sep}_{\mathscr{#1}}\left(#2,\,#3\right)}

\newcommand{\Fun}[2]{\mathrm{Fun}\left({#1},{#2}\right)}

\newcommand*{\FN}{\mathit{FN}}
\newcommand*{\FQ}{\mathit{FQ}}
\newcommand*{\BQ}{\mathit{BQ}}

\let\uhr\upharpoonright
\renewcommand*{\upharpoonright}{\hspace{-.07cm}\uhr\hspace{-.07cm}}
\newtheorem{thm}{Theorem}

\newtheorem{cor}[thm]{Corollary}

\theoremstyle{remark}

\theoremstyle{definition}
\newtheorem{defn}[thm]{Definition}
\newtheoremstyle{axiom1}
{3pt}
{3pt}
{\itshape}
{}
{\bfseries\itshape}
{}
{.5em}
{}
\newtheoremstyle{axiom2}
{3pt}
{3pt}
{\itshape}
{}
{\bfseries\itshape}
{\\}
{.5em}
{}
\theoremstyle{axiom1}
\title{Basic Topological Concepts and a Construction of Real Numbers in Alternative Set Theory}
\author{Kiri Sakahara\thanks{Yokohama National University and Kanagawa University, Kanagawa, Japan.} \and Takashi Sato\thanks{Toyo University, Tokyo, Japan.}}
\date{}
\begin{document}

\maketitle
\begin{abstract}
Alternative set theory (AST) may be suitable for the ones who try to capture objects or phenomenons with some kind of indefiniteness of a border.
While AST provides various notions for advanced mathematical studies, correspondence of them to that of conventional ones are not fully developed.
This paper presents basic topological concepts in AST, and shows their correspondence with those of conventional ones, and isomorphicity of a system of real numbers in AST to that of conventional ones.
\end{abstract}

\section{Introduction}
To represent a collection of human beings who are Charles Darwin's ancestors, what kind of mathematical concepts are appropriate?
Not a set, wrote Vop\v{e}nka in his \textit{Mathematics in the Alternative Set Theory} \cite{ast}, since it has no clear boundary between human beings and apes, both of which are Mr. Darwin's predecessors and in blood relationship within certain steps.
But the former is included in the collection while the latter is not, the border must remain indefinite.
To capture such nebulous collections properly, one has to develop alternative concepts, other than sets, admitting this kind of indefiniteness.

This example is too peculiar.
However, the type of indefiniteness is prevalent when one tries to put together a huge amount of objects and to represent it by a single mathematical entity.
When Vop\v{e}nka set out to rebuild a set-theoretical framework from a different perspective, putting in place this indefiniteness within its framework consistently was one of his main objectives.
The resulting system, the alternative set theory(AST for short), enables this kind of representation possible, especially when one tries to capture objects or phenomenons in the real world.

Vop\v{e}nka \cite{ast} displays various notions of AST, each of which is fundamental to advance mathematical studies.
While formal aspects of the system are relatively detailed, correspondence of them to that of conventional ones are not fully developed.
For example, basic ideas of topological spaces, such as open class or continuity, are not mentioned.
Lack of them may, unfortunately, keep ones who need them most away from these attractive frameworks only because those concepts they are accustomed to are absent.

The present paper aims to fill in this gap.
However, almost all notions given here are simple extension of that of Vop\v{e}nka \cite{ast}, except that those regarding \textit{continua} and \textit{morphisms} are due to Tsujishita \cite{tjst}.
The readers can find relatively compact explanation of AST, such as an axiomatic system of AST, in Vop\v{e}nka and Trlifajov\'{a} \cite{encycro-ast} and Sakahara and Sato \cite{cogjump}.

\section{Preliminaries}
Let us start with constructing a number system in accordance with Vop\v{e}nka \cite{ast}.
The class of \textit{natural numbers} $N$ is defined as:
\[
N\ =\ \left\{x\ ;\
\begin{matrix}
\left(\forall y\in x\right) \left(y\subseteq x\right)\\
\wedge\left(\forall y,z\in x \right) \left(y\in z \vee y=z \vee z\in y\right)
\end{matrix}
\right\},
\]
while the class of \textit{finite natural numbers} $\FN$ consists of the numbers represented by finite sets
\[
\FN \ =\ \left\{x\in N\ ;\
\Fin(x)\right\}
\]
in which $\Fin(x)$ means that each subclass of $x$ is a set.
The class is said to be \textit{countable} iff it is infinite class and any of its initial segment is finite.
One example of countable classes is $\FN$ since all of its initial segments are finite, while $N$ is \textit{uncountable}.

The class of all integers $Z$ and that of all rational numbers are defined respectively as:
\[
Z\ =\ N\cup \left\{ \langle 0,a\rangle;\,
 a\ne 0
\right\}
\qquad\text{and}\qquad
Q\ =\ \left\{
\frac{x}{y}\ ;\ x,\,y\in Z \wedge y\ne 0
\right\}.
\]
$\BQ\subseteq Q$ denotes the class of \textit{bounded rational numbers} and $\FQ\subseteq BQ$ the class of \textit{finite rational numbers}, i.e.,
\[
\BQ=\{x\ ;\ \left(\exists i\in\FN\right)
  \left(|x|\leq i\right)
\}
\qquad\text{and}\qquad
\FQ\ =\ \left\{
\frac{x}{y}\ ;\
x,y\in\FN \wedge y\ne 0
\right\}.
\]
The arithmetic and orders on $Q$ are defined as same as usual.
Real numbers will be defined later.
\medskip

Next, let us proceed to the concepts that are necessary to build topological structures.
A class $X$ is a $\sigma$-class (a $\pi$-class) iff $X$ is the union (the intersection) of a countable sequence of set-theoretically definable classes\footnote{
  A \textit{set-theoretically definable class} is the class defined of the form $\{x: \psi(x)\}$ in which $\psi(\cdot)$ is a set formula, the one which includes set variables only.
  A universal class $V$,  classes of natural numbers $N$ and rationals $Q$ are all set-theoretically definable, while classes of finite natural numbers $\FN$, finite rationals $\FQ$ and bounded rationals $\BQ$ are not since $\Fin$ is not a set-formula.
}.

The basic properties of these classes are listed in the next theorem.
\begin{thm}[The 6th theorem at the Section 5 of Chapter 2 of Vop\v{e}nka \cite{ast}]　\label{pi-sigma}
  \begin{enumerate}[(1)]
  \item $X$ is a $\sigma$-class iff $X$ is the union of a countable ascending sequence of  set-theoretically definable classes.
    $X$ is a $\pi$-class iff $X$ is the intersection of a countable descending sequence of set-theoretically definable classes.
  \item The union of two $\pi$-classes is a $\pi$-class.
    The intersection of two $\sigma$-class is a $\sigma$-class.
  \item The union of a countable sequence of $\sigma$-classes is a $\sigma$-class.
    The intersection of a countable sequence of $\pi$-classes is a $\pi$-class.
  \item If $X$, $Y$ are $\sigma$-classes ($\pi$-classes) then $X\times Y$ is a $\sigma$-class (a $\pi$-class).
  \item If $X$ is a $\sigma$-class (a $\pi$-class) then $\dom(X)\equiv\left\{ u;\langle u,v\rangle\in X\right\}$ is a $\sigma$-class (a $\pi$-class).
  \end{enumerate}
\end{thm}

A class $\doteq$ is a \textit{$\pi$-equivalence} iff $\doteq$ is a $\pi$-class and an equivalence relation.
A sequence $\{R_n;\,n\in \FN\}$ is a \textit{generating sequence} of an equivalence $\doteq$ iff the following conditions hold:
\begin{enumerate}[(1)]
\item For each $n$, $R_n$ is a set-theoretically definable, reflexive, and symmetric relation.
\item For each $n$ and each $x$, $y$, $z$, $\langle x,y\rangle\in R_{n+1}$ and $\langle y,z\rangle\in R_{n+1}$ implies $\langle x,z\rangle\in R_n;\,R_0=V^2$.
\item $\doteq$ is the intersection of all the classes $R_n$.
\end{enumerate}
If $\{R_n;\,n\in \FN\}$ is a generating sequence of $\doteq$ then (i) $\doteq$ is a $\pi$-equivalence, (ii) $R_{n+1}\subseteq R_n$ for each $n$, and (iii) $x\doteq y$ holds iff $\langle x,y\rangle \in R_n$ for each $n$.
It is also true that every $\pi$-equivalence has its generating sequence.

\begin{thm}[The first theorem at the Section 1 of Chapter 3 of Vop\v{e}nka \cite{ast}]
Each $\pi$-equivalence has a generating sequence.
\end{thm}

An equivalence $\doteq$ is said to be \textit{compact} iff for each infinite set $u$ there are $x,y \in u$ such that $x \ne y$ and $x\doteq y$.
A relation $\doteq$ is called an \textit{indiscernibility equivalence} iff $\doteq$ is a compact $\pi$-equivalence.

In the environment of alternative set theory, each indiscernibility equivalence represents a topological structure.
Let us denote this structure as continuum after Tsujishita \cite{tjst}.
A \textit{continuum} is a pair $\mathscr{C}=\langle C,\doteq_C\rangle$ of a set-theoretically definable class $C$ and an indiscernibility equivalence defined on $C$.
$C$ is said to be the \textit{support} of the continuum $\mathscr{C}$.

It will be investigated how the structures of continuum correspond to the standard topological concepts in the following sections.

\section{Open/Closed Classes}
Let us fix one continuum $\mathscr{C}=\langle C,\doteq_C\rangle$.
A class $X$ is a \textit{figure} of $\mathscr{C}$ iff $X$ contains with each $x$ all $y$ which satisfies $x\doteq_C y$.

A \textit{monad} of $x$ is defined as
\[
\mon{C}{x}\ =\
\left\{
y\in C: y\doteq_C x
\right\}.
\]
Evidently, $\mon{C}{x}$ is a figure for each $x$.
A class $X$ is a figure iff the monad of each element of $X$ is a subclass of $X$.

The \textit{figure of} $X$ is defined for each $X$ as follows:
\[
\fig{C}{X} \ =\ \{y\in C:(\exists x\in X)(x\doteq_C y)\}.
\]
Basic properties of figures are given as the following theorem.
\begin{thm}[the second theorem at the Section 2 of Chapter 3 of Vop\v{e}nka \cite{ast}]\label{basic}
For each $X$, $Y$ and $x$ of $\mathscr{C}$ we have the following:
\begin{enumerate}[(1)]
\item $\fig{C}{X}$ is a figure.
\item $X\subseteq Y$ implies $\fig{C}{X}\subseteq \fig{C}{Y}$.
\item If $X\subseteq Y$ and if $Y$ is a figure then $\fig{C}{X}\subseteq Y$.
\item $\fig{C}{X\cup Y}=\fig{C}{X}\cup\fig{C}{Y}$.
\item $\fig{C}{X}\cap\fig{C}{Y}=\emptyset$ iff $\fig{C}{X}\cap Y=\emptyset$.
\item $\mon{C}{x} = \fig{C}{\{x\}}$.
\end{enumerate}
\end{thm}

$X$, $Y$ of $\mathscr{C}$ are \textit{separable}, noted as $\Sep{C}{X}{Y}$, iff there is a set-theoretically definable class $Z$ of $\mathscr{C}$ such that $\fig{C}{X}\subseteq Z$ and $\fig{C}{Y}\cap Z = \emptyset$.
It is evident that $\Sep{C}{X}{Y}$ is symmetric.

The closure of a class $X$, denoted as $\cl{C}{X}$, is defined as follow
\[
\cl{C}{X}\ =\ \left\{
x\in C: \neg\Sep{C}{\{x\}}{X}
\right\}.
\]
The interior of a class $X$ is given dually as $\intrr{C}{X}\ =\ C\setminus\cl{C}{C\setminus X}$.

It is evident that the following equation holds:
\[
\cl{C}{C\setminus M}\ =\ C\setminus \intrr{C}{M},
\]
since
$C\setminus\intrr{C}{ M}=C\setminus\left(C\setminus\cl{C}{C\setminus M}\right)=\cl{C}{C\setminus M}$.
It is also easy to verify that the following equation holds:
\[
\cl{C}{M}\ =\ C\setminus\intrr{C}{C\setminus M}.
\]

Each closure $\cl{C}{X}$ has the properties illustrated in the next theorem.
\begin{thm}[the 5th theorem at the Section 2 of Chapter 3 of Vop\v{e}nka \cite{ast}]\label{clo}
For each $X$, $Y$ of $\mathscr{C}$ we have the following:
\begin{enumerate}[(1)]
\item $\cl{C}{X}$ is a figure.
\item $\fig{C}{X} = \fig{C}{Y}$ implies $\cl{C}{X}=\cl{C}{Y}$.
\item $X\subseteq \cl{C}{X}$.
\item $X\subseteq Y$ implies $\cl{C}{X}\subseteq \cl{C}{Y}$.
\item $\cl{C}{X\cup Y}\ =\ \cl{C}{X}\cup\cl{C}{Y}$.
\end{enumerate}
\end{thm}
\noindent
It is also helpful to remind the next theorem.
\begin{thm}[the 10th theorem at the Section 2 of Chapter 3 of Vop\v{e}nka \cite{ast}]\label{closure}
Let $X$ be a figure of $\mathscr{C}$.
Then the following statements are equivalent:
\begin{enumerate}[(1)]
\item $X$ is the figure of a set $u$.
\item $X$ is a $\pi$-class.
\item $C\setminus X$ is a $\sigma$-class.
\item $X=\cl{C}{X}$.
\item $C\setminus X=\intrr{C}{C\setminus X}$.
\end{enumerate}
\end{thm}
\noindent
A figure $X$ is \textit{closed} in $\mathscr{C}$ iff it has one (and therefore all) of the properties of Theorem \ref{closure}.
Contrastively, a class $X$ is \textit{open} in $\mathscr{C}$ iff $C\setminus X$ is a closed figure in $\mathscr{C}$.

It is worth mentioning that not only closed classes but also open classes are figures.
In general, if $X$ is a figure, so too $C\setminus X$ is.
Suppose that if it is not, then for some element $y\in C\setminus X$, there exists $x\in\mon{C}{y}\cap X$.
Since $X$ is a figure, $y$ must be an element of $X$ just because $y\in\mon{C}{x}\subseteq X$.
It is a contradiction.
Open classes are also figures by the same token, since closed classes are figures.

A figure $X$ is said to be \textit{clopen} in $\mathscr{C}$ iff it has one (thus, both) of the properties of Theorem \ref{closed}.
\begin{thm}[the 6th Theorem at the Section 3 of Chapter 3 of Vop\v{e}nka \cite{ast}]\label{closed}
  For each figure $X$ in $\mathscr{C}$, the following  are equivalent
  \begin{enumerate}[(1)]
  \item $X$ is set-theoretically definable
  \item $X$ and $C\setminus X$ are closed figures.
  \end{enumerate}
\end{thm}

The equivalence $\doteq_\mathscr{C}$ is said to be \textit{totally disconnected} iff it has a generating sequence $(S_i)_{i\in\FN}$ such that $S_i$ is an equivalence for each $i$.
\begin{thm}[The 10th theorem at the Section 3 of Chapter 3 of Vop\v{e}nka \cite{ast}]
  The following are equivalent:
  \begin{enumerate}[(1)]
  \item $\doteq_C$ is set-theoretically definable.
  \item $\mon{C}{x}$ is a clopen figure for each $x\in C$.
  \item The class $V/\doteq_C$ is finite.
  \end{enumerate}
  \end{thm}

Now, it is possible to construct a topological space out of a continuum $\mathscr{C}$.
Let us check the claim with the next theorem.

\begin{thm}
  Given a continuum $\mathscr{C}=\langle C,\doteq_\mathscr{C}\rangle$ in which $C$ is clopen, the space $\langle C,\mathscr{O}\rangle$ is topological where $\mathscr{O}$ is given as
  \[
  \mathscr{O}\ = \ \{\intrr{C}{X}\,:\, X\subseteq C\}
  \]
\end{thm}

\begin{proof}
  It is evident that $C\in\mathscr{O}$ and $\emptyset\in\mathscr{O}$.

  Let $X,Y\in\mathscr{O}$.
  Then, by Theorem \ref{closure}. there exist $x,y\subseteq C$ which satisfy
  \[
  X\ =\ C\setminus\fig{C}{x}\quad\text{ and }\quad Y\ =\ C\setminus\fig{C}{y}
  \]
  Then, by (4) of Theorem \ref{basic},
  \[
  X\cap Y\ =\ C\setminus\left(\fig{C}{x}\cup\fig{C}{y}\right) \ =\ C\setminus\fig{C}{x\cup y}.
  \]
  Since $\fig{C}{x\cup y}$ is a $\pi$-class, $C\setminus\fig{C}{x\cup y}$ is $\sigma$-class,
  so that $C\setminus\fig{C}{x\cup y}=X\cap Y=\intrr{C}{X\cap Y}$ by (3) and (5) of Theorem \ref{closure}, and thus, $X\cap Y\in\mathscr{O}$.

  Lastly, let $(C_i)_{i\in\FN}$ be a sequence of elements of $\mathscr{O}$. Then, for each $i$ there exists a set $c_i\subseteq C$ which satisfies
  \[
  C_i\ =\ C\setminus \fig{C}{c_i}.
  \]
  Then, the following equation holds.
  \[
  \bigcup_{i\in \FN}C_i \ =\ C\setminus\bigcap_{i\in\FN}\fig{C}{ c_i}
  \]
  Since $\bigcap_{i\in\FN}\fig{C}{ c_i}$ is a $\pi$-class by (3) of Theorem \ref{pi-sigma},
  $C\setminus\bigcap_{i\in\FN}\fig{C}{ c_i}$ is $\sigma$-class, so that $C\setminus\bigcap_{i\in\FN}\fig{C}{ c_i} =\bigcup_{i\in \FN}C_i\in\mathscr{O}$ by exactly the same reasoning as the previous case.
\end{proof}

\section{Points and Sequences}

A monad $\mon{C}{x}$ is said to be a \textit{point} of $\mathscr{C}$ and $x$ its \textit{position}.
A position is not uniquely determined since each point contains multiple different elements in general, and each elements are equally qualified to be its position.

For any subclass $A$  of $C$, a point $\mon{C}{x}$ is said to be an \textit{accumulation point} of $A$ iff it satisfies that
\[
x\in \cl{C}{A\setminus \mon{C}{x}}\quad\text{or}\quad
\mon{C}{x}\subseteq \cl{C}{A\setminus\mon{C}{x}}.
\]
When it is not but the point of $A$, it is said to be an \textit{isolation point} of $A$.

Let $\left(\mon{C}{a_i}\right)_{i\in\tau}$ and $(a_i)_{i\in\tau}$ be sequences of points and their positions of $\mathscr{C}$ respectively, and $\left(R_i\right)_{i\in\FN}$ be a generating sequence of $\doteq_\mathscr{C}$.
A sequence $(\mon{C}{a_i})_{i\in\FN}$ of $C$ is said to \textit{converge to} a point $\mon{C}{x}$ iff
\[
\left(
\forall k\in\FN
\right)
\left(
\exists i\in\FN
\right)
\left(
\forall j>i
\right)
\left(
\langle a_j,x\rangle\in R_k
\right).
\]
It is simply denoted as $\lim_{i\in\FN} \mon{C}{a_i}\ = \ \mon{C}{x}$ if it converges to $\mon{C}{x}$.

Let $Z_n(a)$ denote an \textit{image} of $a\in C$ with $R_n$, defined as
  \[
  Z_n(a)\ \equiv\ \left\{
  x\in C: \langle a,x\rangle\in R_n
  \right\}.
  \]
  Then, the next theorem holds.

\begin{thm}\label{acc}
  Let $\doteq_C$ be a non totally disconnected indiscernibility equivalence and $A\subseteq C$.
  A point $\mon{C}{a}$ is an accumulation point of $A$ iff it is a limit of some converging sequence $\left(\mon{C}{a_i}\right)_{i\in\FN}$ consisting only of mutually different points of $A$.
\end{thm}
\begin{proof}
  Let   $(R_i)_{i\in\FN}$ be a generating sequence of $\doteq_\mathscr{C}$ and $Z_n(a)$ be an image of $a$.
  Since $\doteq_C$ is not totally disconnected, every element $R_n$ of its generating sequence is strictly bigger than $\doteq_C$, so that $Z_n(a)$ is also strictly bigger than $\mon{C}{a}$ for all $n\in\FN$.

  Suppose there exists a set-theoretically definable class $Z$ which separates $\mon{C}{a}$ and $A\setminus\mon{C}{a}$, that is, which satisfies $\mon{C}{a}\subseteq Z$ and $Z\cap A\setminus\mon{C}{a}=\emptyset$.
  Since $\mon{C}{a}=\bigcap_{i\in\FN} Z_i(a)$, there exists $j\in\FN$ such that $Z_k(a)\subseteq Z$ for all $k>j$.
  Since $\left(\mon{C}{a_i}\right)_{i\in\FN}$ converges to $\mon{C}{a}$, there exists $a_i\in Z_k(a)\setminus\mon{C}{a}$ for every $k>j$.
  It means $Z\cap A\setminus\mon{C}{a}\ne\emptyset$, since $a_i\in A\setminus\mon{C}{a}$. It is a contradiction.

  Conversely, suppose $\mon{C}{a}$ is an accumulation point of $A$.
  Then, it holds that $Z_n(a)\cap A\setminus\mon{C}{a}\ne\emptyset$ for all $n\in\FN$.
  Since $\doteq_C$ is not totally disconnected,
  there exists $a_0\in Z_{n_0}(a)$ which satisfies $a_0\not\doteq a$.
  It is also evident that there exists $n_1\in\FN$ which satisfies $n_1>n_0$ and $Z_{n_1}(a)\cap\mon{C}{a_0}=\emptyset$, otherwise $a_0$ must be indiscernible with $a$ by definition of $\doteq_C$,
  and that there exists $a_1\in Z_{n_1}(a)$ which satisfies $a_1\not\doteq a$.
  By repeating it countably many times, we can get a sequence of points of $\mathscr{C}$ as $\left(\mon{C}{a_i}\right)_{i\in\FN}$ which converges to $\mon{C}{a}$ and its elements are mutually different since each pair of positions $a_i$ and $a_{i+1}$ is separated by set-theoretically definable class $Z_{n_{i+1}}(a)$.
\end{proof}

This theorem guarantees the well-known property of closed classes as shown in the next theorem.
\begin{thm}\label{clo}
Every converging sequence of $A$ has its limits in $A$ iff $A$ is closed.
\end{thm}
\begin{proof}
  Let $(\mon{C}{a_i})_{i\in\FN}$ be a converging sequence of $A$.
  When it has a subsequence consisting only of mutually different points, the limit is an accumulation point of $A$ by Theorem \ref{acc}.
  Otherwise, the limit is an isolation point.
  Both limits are included in $A$ since $A$ is closed.
  Converse is also true.
\end{proof}

Points of topological structures of AST satisfy the following separability property.

\begin{thm}\label{sep}
Every pair of different points $\mon{C}{x}\ne\mon{C}{y}$ is separable.
\end{thm}
\begin{proof}
  Since $x\not\doteq y$, there exists $n\in\FN$ which satisfy $\langle x,y\rangle\notin R_n$.
  It implies that  $\mon{C}{x}\subseteq Z_n(x)$ and $Z_n(x)\cap\mon{C}{y}=\emptyset$.
\end{proof}

\section{Neighborhoods}

A class $X\subseteq C$ is a \textit{neighborhood} of a point $\mon{C}{x}$ iff
\[
x\in\intrr{C}{X}.
\]
The open class which contains $x$ is called an \textit{open neighborhood} of $x$.
A \textit{complete system of neighborhoods} $\mathscr{V}(x)$ for a point $\mon{C}{x}$ is the collection of all neighborhoods for the point $x$.

\begin{thm}
  Let $\doteq_C$ be a non totally disconnected indiscernibility equivalence, $\left(\mon{C}{a_i}\right)_{i\in\FN}$ be a converging sequence, and $\mathscr{V}(x)$ be a complete system of neighborhoods.
  A point $\mon{C}{a}$ is a limit of $\left(\mon{C}{a_i}\right)_{i\in\FN}$ iff for each $V\in\mathscr{V}(a)$ there exists $i\in\FN$, $\{a_j:j>i\}\subseteq V.$
\end{thm}
\begin{proof}
  Let $(R_n)_{n\in\FN}$ be a generating sequence of $\doteq_\mathscr{C}$ and $Z_n(a)$ be an image of $a$ with $R_n$.
  Then, for every neighborhood $V\in\mathscr{V}(a)$, there exists $\ell\in\FN$ which satisfies $Z_n(a)\subseteq V$ for all $n>\ell$.
  Since $(\mon{C}{a_i})_{i\in\FN}$ converges to $\mon{C}{a}$, there exists $a_{i_n}\in Z_n(a)$ for each $Z_n(a)$.

  Conversely, since each $Z_n(a)$ is a neighborhood of $a$, or $Z_n(a)\in\mathscr{V}(a)$, there exists $n_i\in\FN$ which satisfies $\{a_{j}:j>n_i\}\subseteq Z_{n}(a)$ for each $n\in\FN$.
  It implies $(\mon{C}{a_i})_{i\in\FN}$ converges to $\mon{C}{a}$.
\end{proof}

\section{Compactness}
A family $\mathscr{A}$ of classes is a \textit{cover} of a class $X$ iff $X\subseteq\cup\mathscr{A}$.
A continuum $\mathscr{C}$ is \textit{compact} iff every  
open cover of $C$ has its finite subcover.
A class $X\subseteq C$ is \textit{compact} iff every open cover of $X$ has its finite subcover.

\begin{thm}
A continuum $\mathscr{C}$ is compact iff every countable sequence has a converging subsequence and its limit in ${C}$.
\end{thm}
\begin{proof}
  Suppose $\mathscr{C}$ is compact and $(\mon{C}{a_i})_{i\in\FN}$ 
  has no converging subsequences.
  Then, $F=\{\mon{C}{a_i}:i\in\FN\}$ is a closed class by Theorem \ref{clo} since it has no accumulation points.
  Let $O=C\setminus F$, then, $O$ is open and
  \[
  a_i\notin O\quad \text{ for all }i\in\FN.
  \]
  Since each point is discernible, there exists a sequence of naturals $(n_i)_{i\in\FN}$ and set-theoretically definable classes $Z_{n_i}(a_i)$ which are pairwise disjoint.
  Then, $C$ is covered by the union of the following classes
  \[
  C\ =\ O\cup \bigcup_{i\in\FN} O_i\qquad \text{where }\ O_i=\intrr{C}{Z_{n_i}(a_i)}.
  \]
  However, no finite subcover exists.
  It contradicts with the compactness of $\mathscr{C}$.

  Conversely, suppose $\mathscr{C}$ is not compact and every countable sequence has converging subsequences.
  Let $C=\bigcup_{i\in\FN} O_i$ be a countable cover of $C$ but has no finite subcover.
  Choose for each $n\in\FN$
  \[
  a_n\notin \bigcup_{i\in n+1}O_i.
  \]
  Then, $\left\{a_i\right\}_{i\in\FN}$ has an accumulation point.
  It implies there exist a strictly increasing function $F:\FN\rightarrow \FN$ and subsequence $\left(a_{F(j)}\right)_{j\in\FN}$ which satisfies $\lim_{j\rightarrow\infty}{\mon{C}{a_{F(j)}}}=\mon{C}{a}$.
  Since $C=\bigcup_{i\in\FN}O_i$, there must be $k\in\FN$ which satisfies $\mon{C}{a}\subseteq O_{k}$.
  Since $\left(a_{F(j)}\right)_{j\in\FN}$ converges to $\mon{C}{a}$, there must be $\ell\in\FN$ which satisfies $\ell>k$ and $a_m\in O_k\subseteq  \bigcup_{i\in m+1} O_i$ for all $m>\ell$.
  But, $a_m\notin \bigcup_{i\in m+1}O_i$ by definition.
  It is a contradiction.
\end{proof}

Let us, next, introduce an \textit{$R$-net}.
A class $X\subseteq C$ is an $R$-net iff there are no distinct element $x,y\in X$ such that $\langle x,y\rangle\in R$.
$X$ is \textit{maximal $R$-net on $C$} iff $X\subseteq C$ and for each $y\in C$ there is an $x\in X$ such that $\langle x, y\rangle\in R$.

A relation $R$ is an \textit{upper bound} of an equivalence $\doteq$ iff $R$ is symmetrical, set-theoretically definable and $\doteq$ is a subclass of $R$, i.e., $x\doteq y$ implies $\langle x,y\rangle\in R$ for each $x,y\in C$.

\begin{thm}[the third theorem at the Section 1 of Chapter 3 of Vop\v{e}nka \cite{ast}]\label{rnet}
    Let $\doteq$ be a compact equivalence and let $R$ be its upper bound.
    Then there is a finite number $n$ such that, for each $R$-net $X$, $X$ is subvalent to $n$, i.e., there is a one-one mapping of $X$ onto a subclass of $n$.
\end{thm}
$X\precsim n$ denotes that $X$ is subvalent to $n$ hereafter.

The following theorem implies the equivalence between compactness and being closed in an AST environment.

\begin{thm}\label{comclo}
A figure $X$ of $\mathscr{C}$ is compact iff it is closed.
\end{thm}
\begin{proof}
  Suppose that $X$ is compact.
  Let us choose $x\in X$ and $y\in C\setminus X$ arbitrarily.
  Since $X$ is a figure and every points are separated by Theorem \ref{sep}, there exists $n\in\FN$ and set-theoretically definable classes $Z_n(x)$ and $Z_n(y)$ which are mutually disjoint and separate $\mon{C}{x}$ and $\mon{C}{y}$.
  Then, open classes $O_n(x)=\intrr{C}{Z_n(x)}$ and $O_n(y)=\intrr{C}{Z_n(y)}$ are mutually disjoint.
  Now, let us fix $y\in C\setminus X$ and choose large enough $n\in\FN$ so that the following inclusion holds:
  \[
  X\subset \bigcup_{x\in X} \{O_n(x): O_n(y)\cap O_n(x)=\emptyset\}.
  \]
  Since $X$ is compact, there exists a finite subset $m$ of $X$ which covers $X$
  \[
  X\subset \bigcup_{x\in m} \{O_n(x): O_n(y)\cap O_n(x)=\emptyset\}.
  \]
  For each $Z_n(x)$ is set-theoretically definable, $\bigcup_{x\in m}\{Z_n(x):Z_n(y)\cap Z_n(x)=\emptyset\}$ too is set-theoretically definable, so that it separates $\{y\}$ and $X$.
  Since $y$ is chosen arbitrarily, it implies that $X$ is closed.

  Suppose, conversely, that $X$ is closed but not compact.
  Then, it has a countable sequence with no converging subsequences.
  Let us choose one such sequence of points $(\mon{C}{a_i})_{i\in\FN}$ arbitrarily.
  Now, let $X$ be a maximal $R$-net which satisfies $X\precsim n$ for some $n\in\FN$.
  Then, there must be at least one element $x_i\in X$ whose image $Z_i(x_i)$ include infinite points for each $i\in\FN$, that is, $\bigcup_{j\in J_i}\mon{C}{a_j}\subseteq Z_{i}(x_{i})$ where $J_i$ is a countable subclass of $\FN$. 
  Let us define a function $F(i)=j$ which satisfies $a_j\doteq x_{i}$.
  Then, $(\mon{C}{a_{F(j)}}_{j\in \FN}$ is a subsequence of $(\mon{C}{a_i})_{i\in\FN}$ and converges to $\mon{C}{a}\in\bigcap_{i\in\FN} Z_i(x_{i})$.
  It is a contradiction.
\end{proof}

\section{Connectedness}

A set $u$ is \textit{connected} in $\mathscr{C}$ iff for each  nonempty proper subset $v$ of $u$ there exist $x\in v$ and $y\in u\setminus v$ which satisfy $x\doteq_C y$.
To notify a set $u$ is connected, let us denote $\Cntd{1}{u}$.
A figure $X$ is connected iff for each $x,y\in X$ there is a connected set $u\subseteq X$ which satisfies $x,y\in u$.
A continuum $\mathscr{C}$ is \textit{connected} iff for all figures of $\mathscr{C}$ is connected.


\begin{thm}
  A continuum $\mathscr{C}$ is connected iff $\mathscr{C}$ has no clopen figures besides $C$ and $\emptyset$.
\end{thm}

\begin{proof}
  Suppose $\mathscr{C}$ has a clopen figure $X$, then, $C\setminus X$ is set-theoretically definable, thus, clopen.
  It contradicts with the $\mathscr{C}$'s connectedness.

  Conversely, suppose that $\mathscr{C}$ is not connected, then, there exists at least one clopen figure $X$.
  It is a contradiction.
\end{proof}

\begin{thm}
  Let $\mathscr{C}$ be a continuum and $X\subseteq C$ be a set-theoretically definable class.
  $X$ is connected iff there exists no clopen figure $Y$ in $\mathscr{X}=\langle X,\doteq_C\,\upharpoonright X\rangle$ which satisfies $X\cap Y\ne\emptyset$ and $X\setminus Y\ne\emptyset$.
\end{thm}

\begin{proof}
  Suppose there exists a clopen figure $Y$ which satisfies $X\cap Y\ne \emptyset$ and $X\setminus Y\ne\emptyset$.
  Then, for all $x\in X\setminus Y\subseteq C\setminus Y$ and $y\in X\cap Y$, it is satisfied that $\neg(x\doteq_C y)$.
  It contradicts with $X$'s connectedness.

  Conversely, suppose $X$ is not connected.
  Then, there exists a pair of points $x,y\in X$ which has no subset $u\subseteq X$ which is connected and include them.
  It implies that $\{x\}$ and $\{y\}$ are separable in $\mathscr{X}$, so that there exists a set-theoretically definable figure $Y$ in $\mathscr{X}$ which satisfies $\cl{X}{x}\subseteq Y$ and $\cl{X}{y}\cap Y=\emptyset$.
  Since $X$ is set-theoretically definable, both $Y\cap X$ and $X\setminus Y$ are also set-theoretically definable and nonempty.
  It is a contradiction.
\end{proof}

\begin{thm}
  Let $\mathscr{C}$ be a continuum and $X$ be a closed figure.
  Then $X$ is connected iff there exists no pair of closed figures $Y_1$ and $Y_2$ such that
  \[
  X\subseteq Y_1\cup Y_2\qquad
  Y_1\cap Y_2\cap X=\emptyset\qquad
  Y_1\cap X\ne\emptyset\qquad
  Y_2\cap X\ne\emptyset.
  \]
\end{thm}
\begin{proof}
  Suppose $Y_1$ and $Y_2$ satisfy the condition, there exist mutually disjoint sets $y_1,y_2\subseteq X$ which satisfy $\fig{C}{y_1}=Y_1\cap X$ and $\fig{C}{y_2}=Y_2\cap X$.
  Since $\fig{C}{y_1}\cap\fig{C}{y_2}=\emptyset$, for all $a\in y_1$ and $b\in y_2$, $\neg(a\doteq_C b)$ holds.
  It contradicts with $X$'s connectedness.

  Conversely, if $X$ is not connected, there exists mutually disjoint sets $y_1$ and $y_2$ which satisfies $y_1\cup y_2=x$ where $\fig{C}{x}=X$ and $\neg(a\doteq_C b)$ for all $a\in y_1$ and $b\in y_2$.
  Then, $Y_1=\fig{C}{y_1}$ and $Y_2=\fig{C}{y_2}$ satisfy all the four conditions.
\end{proof}

\begin{cor}
If continuum $\mathscr{C}$ has no clopen figures, there exist connected sets for any pair of sets $x,y\in C$.
\end{cor}

\section{Metric Spaces}
Let $\mathscr{C}$ be a continuum.
A metric $d$ is a function $d: C\times C\rightarrow Q$ which satisfies
\begin{enumerate}[(1)]
\item $d(x,y)\geq 0$ for any $x,y\in C$,
\item $d(x,y)=0$ iff $x=y$ for all $x,y\in C$,
\item $d(x,y)=d(y,x)$ for all $x,y\in C$,
\item for all $x,y$, and $z\in C$, the inequality $d(x,z)\leq d(x,y)+d(y,z)$ holds.
\end{enumerate}
A pair $\langle C,d\rangle$ is a \textit{metric space} where $d$ is a metric of $C$.
A \textit{distance} between $x$ and $y$ of $C$ is given as $d(x,y)$.

Let $a\in C$ and $e>0$ be a finite rational number.
Then, a \textit{ball of radius $e$ centered at the position $a$} is defined as
\[
B(a;e)\ =\ \bigcup_{i\in\FN}\left\{x\,:\, \left( x\in C\right)\wedge\left(d(a,x)<e-\frac{1}{2^i}\right)\right\}
\]
\begin{thm}
  There exists an indiscernibility equivalence $\doteq_B$ which makes $B(a;e)$ an open class for any $a\in C$ and $e\in\FQ$.
\end{thm}
\begin{proof}
  Let $\left(R_n\right)_{n\in\FN}$ be a sequence of set-theoretically definable class as
  \[
  R_n\ =\ \left\{\langle a,b\rangle\in C\times C: \left(d(a,b)<e-\frac{1}{2^n}\right)\vee(d(a,b)>2^n)\right\}.
  \]
  Then, $\cap_{n\in\FN} R_n$ forms an indiscernibility equivalence $\doteq_B$.

  Let $\mathscr{C}_B$ be a continuum $\langle C,\doteq_B\rangle$.
  Given $a\in C$ and $e\in\FQ$, define a class $A=\{x\in C:d(a,x)< e\}$.
  Then, the ball $B(a;e)$ is given as $\intrr{C_\text{B}}{A}$.
\end{proof}

\section{Real Continua}
Let $\mathscr{R}=\langle Q,\doteq\rangle$ be a \textit{real continuum} where $\doteq$ is a indiscernibility equivalence defined as $\doteq\ \equiv\bigcap_{n\in\FN} R_n$ in which
\[ R_n=\left\{\langle a,b\rangle\in Q\times Q:\left(d(a,b)<\frac{1}{2^n}\right)\vee\left(d(a,b)>2^n\right)\right\},\]

and let $R=\BQ/\doteq$ be a class of real numbers.

Notice that a real continuum $\mathscr{R}$ contains $\infty\equiv \mon{}{\alpha}$ and $-\infty\equiv\mon{}{-\alpha}$ for $\alpha\in N\setminus\FN$ as its elements, while $R$ doesn't.

For each bounded rational $q\in \BQ$ there exists a unique real $\mon{}{q}\in R$.
Let us denote such a real simply as $q\in R$.
Conversely, for each real $r\in R$ there exists a bounded rational number $q\in \BQ$ which satisfies $r= \mon{}{q}$.
Let us also denote such a rational simply as $r\in \BQ$ for notational ease, hereafter.
Let $a, b\in R$ which satisfy $a<b$.
Then, the real intervals $(a,b]$, $(a,b)$, $[a,b]$ and $[a,b)$ are given respectively as
\begin{eqnarray*}
  (a,b]\ &=&\ \cl{}{\{q\in Q:a<q<b\}}\setminus\mon{}{a}\\[.18cm]
  &=&\ \intrr{}{\{q\in Q:a<q<b\}}\cup\mon{}{b}\\[.18cm]
  (a,b)\ &=&\  \intrr{}{\{q\in Q:a<q<b\}}\\[.18cm]
  [a,b]\ &=&\ \cl{}{\{q\in Q:a<q<b\}}\\[.18cm]
  [a,b)\ &=&\ \cl{}{\{q\in Q:a<q<b\}}\setminus\mon{}{b}\\[.18cm]
    &=&\ \intrr{}{\{q\in Q:a<q<b\}}\cup\mon{}{a}
\end{eqnarray*}

To make sure that this construction of $R$ is \textit{really} identical with that of real numbers, let us next examine its characteristics of algebraic structures.

Arithmetic operations on $R$ are given as same as that on $Q$, that is: for all $a,b\in \BQ$
\begin{eqnarray*}
\mon{}{a}+\mon{}{b} & =& \mon{}{a+ b}\\
\mon{}{a}\cdot\mon{}{b} & =& \mon{}{a\cdot b}.
\end{eqnarray*}
An ordered relation is also given as the same manner, that is: for all $a,b\in \BQ$
\[
\mon{}{a}\leq\mon{}{b}\ \Leftrightarrow\ a\leq b.
\]
It is evident that $(R,\leq,+,\cdot)$ constitutes an ordered field.

\begin{defn}
  $D\subseteq R$ is called \textit{Archimedean} iff for each $x\in D$ there exists an $n\in\FN$ such that $x<\mon{}{n}$.
  Otherwise, $D$ is called \textit{non-Archimedean}.
\end{defn}

\begin{thm}
$R$ is an Archimedean.
\end{thm}
\begin{proof}
  For any given $a\in R$ there exists $q\in\BQ$ such that $a=\mon{}{q}$.
  Since $q$ is bounded there always exists $n\in \FN$ in which $n>q$.
  It implies that $a<\mon{}{n}$.
\end{proof}

Let $D\subseteq R$. 
The element $\ell\in D$ is an \textit{upper bound} of the nonempty class $A\subseteq D$ iff $x\leq \ell$ for all $x\in A$.
Moreover, if no $m\in D$ for which $m<\ell$ is an upper bound of $A$, $\ell$ is said to be the \textit{least upper bound} of $A$.

\begin{defn}
  $D$ is \textit{complete} iff every nonempty subclass of $D$ that has an upper bound has a least upper bound.
\end{defn}

\begin{thm}
  $R$ is a complete ordered field.
\end{thm}
\begin{proof}
  Let us check that $R$ is complete.
  Let $C\subseteq R$, $C\ne \emptyset$ and $C$ has an upper bound in $R$.

  Let $B=\{b\in Q: \mon{}{b}\text{ is an upper bound of } C\}$.
  Since $R$ is Archimedean, for all $x\in C$ there exists $q\in Q$ and $n\in\FN$ satisfying that $x=\mon{}{q}<\mon{}{n}$,
  so that $B\ne\emptyset$.
  Conversely, let $A=Q\setminus B$.
  Then $A$ is not empty too since for every $x\in C$ there exists $q\in A$ such that $x=\mon{}{q}$.

  Choose $a\in A$ and $b\in B$ arbitrarily.
  Let $c_0=b$, and
  \[
  c_{i+1}\ =\
  \begin{cases}
    c_i+\frac{a-b}{2^{i+1}} & \text{ if}\quad c_i+\frac{a-b}{2^{i+1}}\in B\\
    c_i & \text{ otherwise}
  \end{cases}
  \quad\text{ for all } i\in\tau+1\in N\setminus\FN.
  \]
  Then, since $c_{\alpha+1}- c_\alpha\leq\frac{a-b}{2^{\alpha+1}}\doteq 0$ for all $\alpha\in\tau\setminus\FN$, the equation holds:
  \[
  \mon{}{c_\alpha}=\mon{}{c_\tau}.
  \]
  Suppose that there exists $z\in B$ which satisfies $z<c_\tau$ and $z\not\doteq c_\tau$.
  Then, there exists $\ell\in\FN$ which satisfies
  \[
  \frac{a-b}{2^{\ell}}\ <\ z-c_\tau \ \leq\ \frac{a-b}{2^{\ell+1}}.
  \]
  It implies that there exists $m\leq \ell$ which satisfies
  $c_{m}+\frac{a-b}{2^{m+1}}\in B$ but $c_{m+1}=c_m$.
  It contradicts with the way $c_i$ is built.

  Therefore, it is concluded that $\mon{}{c_\tau}$ is the least element of $B/\hspace{-.144cm}\doteq$, that is, the least upper bound of $C$.
\end{proof}

The theorem assures us that the construction of $R$ in AST is isomorphic to, say, $\mathbb{R}$ of \textit{ZFC}.

\section{Morphisms}
Let $\mathscr{C}_1$ and $\mathscr{C}_2$ be continua, and $F:C_1\rightarrow C_2$ be a function.
$F$ is \textit{continuous} iff for all $a,b\in C_1$ which are mutually indiscernible, that is, $a\doteq_{C_1} b$, $F(a)$ and $F(b)$ are also indiscernible, $F(a)\doteq_{C_2} F(b)$.
Two continuous functions $F$ and $G$ are \textit{indiscernible}, denoted simply as $F\doteq G$, if for all $x\in C_1$, $F(x)\doteq G(x)$ follows\footnote{
  To be precise,  put
  \[
  r_i\ \equiv\
  \left\{\langle f,g\rangle;\,
  \left(f,g\in \Fun{c_1}{c_2}\right)
  \wedge\left(\forall x\in c_1\right)\left(f(x)-g(x)\leq\frac{1}{2^i}\right)
  \right\},
  \]
  in which $C_i\subseteq c_i$ for $i=1,2$ when $C_1$ is a semiset, otherwise $H(C_i)\subseteq c_i$ for a given similarity endomorphism $H:V\rightarrow D$ in which $D$ is a semiset (for the definition and its existence see p.111 of Vop\v{e}nka \cite{ast}), and $\doteq_{c_2^{c_1}}\ \equiv\ \bigcap_{i\in\FN} {r}_i$.

  For each pair of functions $F,G\in\Fun{C_1}{C_2}$,
  let us denote $F\doteq G$ iff there exists their (or their similar classes') prolonged sets of functions $f,g\in\Fun{c_1}{c_2}$ which satisfies $F=f\upharpoonright C_1$ (or $H(F)=f\upharpoonright H(C_1)$), $G=f\upharpoonright C_1$ (or $H(G)=g\upharpoonright H(C_2)$) and $f\doteq_{c_2^{c_1}} g$.
}.

The \textit{morphism} $\mathscr{F}$ between two continua is defined as follows\footnote{Contrary to the framework of Tsujishita \cite{tjst}, in which the morphisms are not guaranteed to be classes, they are in AST.}.
\begin{enumerate}[(1)]
\item A \textit{morphism} from $\mathscr{C}_1$ to $\mathscr{C}_2$ is a monad $\mon{}{F}$ denoted simply as $[F]$ for some continuous function $F$ from $C_1$ to $C_2$.
\item If $\mathscr{C}_i$ ($i=1,2$) are continua, the notation  $\mathscr{F}:\mathscr{C}_1\rightarrow \mathscr{C}_2$ means that $\mathscr{F}$ is a morphism from $\mathscr{C}_1$ to $\mathscr{C}_2$ .
\item If $\mathscr{F}$ is a morphism, then the expression $G\in \mathscr{F}$ means $G\doteq F$ in which $\mathscr{F}=[F]$.
  If $G\in \mathscr{F}$, we say that the morphism $\mathscr{F}$ is \textit{represented by} $G$ and $G$ \textit{represents} $\mathscr{F}$.
\item If $\mathscr{F}$ and $\mathscr{G}$ are morphisms represented respectively by $F$ and $G$, then the expression $\mathscr{F}=\mathscr{G}$ means $F\doteq G$.
\end{enumerate}
It is essential that morphisms are defined as the monads of continuous function.
When a set-theoretically definable function $F:C_1\rightarrow C_2$ has an indiscernible gap at a position $x\in C_1$, that is, $\neg(F(x)\doteq F(y))$ for some $y\doteq x$, its value at its point $\mon{C_\text{1}}{x}$ cannot be determined uniquely since $x\doteq_{C_1} y$ but $F(y)\not\doteq_{C_2} F(x)$, thus, $\mon{C_\text{2}}{F(x)}\cap \mon{C_{\text{2}}}{F(y)}=\emptyset$.

It is also worth mentioning that continuity of $F$ cannot guarantee its morphism $\mathscr{F}$'s \textit{continuity}, which will be defined later, since it does not guarantee continuous change at its value outside its monad.
For example, the morphism $\mathscr{F}:\mathscr{R}\rightarrow\mathscr{R}$ is not continuous.
\[
\mathscr{F}(x)\ =\ \begin{cases}
1 & \text{ if }\quad x=0\\
0 & \text{ otherwise}
\end{cases}
\]
but any function $F$ which satisfies $[F]=\mathscr{F}$ is continuous since for each $q\in Q$, it is satisfied that $F(q)=0$ except $F(q)=1=F(0)$ when $q\doteq 0$.
To capture morphism's continuity, stronger conditions are needed.

\section{Continuous Morphisms}
Let $\mathscr{F}:\mathscr{C}_1\rightarrow\mathscr{C}_2$ be a morphism from $\mathscr{C}_1$ to $\mathscr{C}_2$.
For simplicity, let $F$ denote a representative of $\mathscr{F}=[F]$ hereafter.

A morphism $\mathscr{F}:\mathscr{C}_1\rightarrow\mathscr{C}_2$ is \textit{continuous} iff
\begin{equation}\label{connect}
  \left(\forall u\subseteq \dom{F}\right)
  \left(
  \Cntd{1}{u}\Rightarrow\Cntd{2}{{F}(u)}
  \right).
\end{equation}

Another way of defining continuity can be given by way of a motion.
A function $\delta_1$ is a \textit{motion of a position} in a time $\theta$ of $\mathscr{C}_1$, in which $\theta\in N$, iff $\dom(\delta_1)=\theta+1$ and for each $\alpha<\theta$, $\delta_1(\theta)\doteq_1 \delta_1(\theta+1)$.
If $\delta_1$ is a motion of a position then $\rng(\delta_1)$ is called a \textit{trace} of $\delta_1$.

Then, a morphism $\mathscr{F}:\mathscr{C}_1\rightarrow\mathscr{C}_2$ is \textit{continuous} iff
\begin{equation}\label{trace}
  \left(\forall \delta_1\right)
  \left(
  \left(\rng{(\delta_1)}\subseteq \dom{(F)}\right)
  \Rightarrow
  \left(
  {F}(\delta_1) \text{ is a motion of a position of $\mathscr{C}_2$}
  \right)
  \right).
\end{equation}

It is easy to examine these two conditions (\ref{connect}) and (\ref{trace}) are equivalent.
To make sure that it is true, two theorems of Vop\v{e}nka at the Section 1  of Chapter 4 are useful.

\begin{thm}[the first Theorem of Vop\v{e}nka \cite{ast} at Section 1 of Chapter 4]\label{vop41-1}
The trace of a motion of a position is a connected set.
\end{thm}
\begin{thm}[the second Theorem of Vop\v{e}nka \cite{ast} at Section 1 of Chapter 4]\label{vop41-2}
For each nonempty connected set $u$ there is a motion of a position such that $u$ is the trace of $\delta$.
\end{thm}

As easily seen by Theorem \ref{vop41-1}, (\ref{trace}) $\Rightarrow$ (\ref{connect}) is evident.
 (\ref{connect}) $\Rightarrow$ (\ref{trace}) is also evident by Theorem \ref{vop41-2}.

Let us next examine relationship between motions and sequences.
It is easily verified that within a connected continuum these two concepts coincide.
\begin{thm}\label{poseq}
  Let $\mathscr{C}$ be a connected continuum.
  Then, following two conditions are equivalent:
  \begin{enumerate}[(1)]
  \item there exists a motion $\delta:\tau+1\rightarrow C$ which satisfies $\delta(0)=a_0$ and $\delta(\tau)=a$,
  \item there exists a sequence $(\mon{C}{a_i})_{i\in\FN}$ with its limit $\lim_{i\in\FN}\mon{C}{a_i}=\mon{C}{a}$.
  \end{enumerate}
\end{thm}
\begin{proof}
  To show (1) $\Rightarrow$ (2), arbitrarily choose one position $a_i\in Z_i(a)\cap\rng{(\delta)}$ for each $i\in\FN$.
  Then, the sequence $\left(\mon{C}{a_i}\right)_{i\in\FN}$ converges to $\mon{C}{a}$ since for each $i\in\FN$ it has a corresponding position $a_i$ satisfying $\langle a_i,a\rangle\in R_i$.

  To show a converse case is straightforward.
  Since $\mathscr{C}$ is connected, between any two positions, there exists motions.
  Thus, motion from $a_0$ to $a$ always exists.
\end{proof}
By Teorem \ref{poseq}, it is equivalent to define continuity by way of converging sequences:
a morphism $\mathscr{F}:\mathscr{C}_1\rightarrow\mathscr{C}_2$ is \textit{continuous} iff for all converging sequences, say $\lim_{i\in\FN}\left(\mon{C_\text{1}}{a_i}\right)=\mon{C_\text{1}}{a}$, the following holds:

\[ \lim_{i\in\FN}\mon{C_\text{2}}{F(a_i)}=\mon{C_\text{2}}{F(a)}. \]
Or equivalently,
\[
(\forall k\in\FN)
(\exists j\in\FN)
((\langle a_i,a\rangle\in R_{1_j})\Rightarrow (\langle F(a_i),F(a)\rangle\in R_{2_k})).
\]

The well known property of continuity described below also holds.

\begin{thm}\label{cont}
  Let $\mathscr{C}_1=\langle C_1,\doteq_1\rangle$ and $\mathscr{C}_2=\langle C_2,\doteq_2\rangle$ be two continua, $C_1$ and $C_2$ be set-theoretically definable class, and $\mathscr{F}:\mathscr{C}_1\rightarrow\mathscr{C}_2$ be a morphism.
  Then, the following three conditions are equivalent.
  \begin{enumerate}[(i)]
  \item For any open class $X$ of $\mathscr{C}_2$, $F^{-1}(X)$ is also an open class of $\mathscr{C}_2$.
  \item For any closed figure $Y$ of $\mathscr{C}_2$, ${F}^{-1}(Y)$ is also a closed figure of $\mathscr{C}_2$.
    \item $\mathscr{F}$ is continuous.
  \end{enumerate}
\end{thm}
\noindent

\begin{proof}
  To prove (i) $\Rightarrow $ (ii), let us remind that for every $X_2\subseteq C_2$, the following equation holds.
  \begin{equation}\label{setminus}
  {F}^{-1}\left(C_2\setminus X_2\right)\ =\ C_1\setminus{F}^{-1}\left(X_2\right)
  \end{equation}
  Given that (i) holds and $X_2$ is open, then ${F}^{-1}(X_2)$ is also open.
  Since  $X_2$ is open and $C_2$ is set-theoretically definable, $C_2\setminus X_2$ is closed, so too is $C_1\setminus {F}^{-1}(X_2)$.
  It implies that ${F}^{-1}(C_2\setminus X_2)$ is closed by the equation (\ref{setminus}).

  (ii) $\Rightarrow$ (i) follows by a similar argument.

  To prove (iii) $\Rightarrow$ (ii), let $Y$ be a closed figure of $\mathscr{C}_2$.
  Then, there exists $v\subseteq Y$ which satisfies $\fig{2}{v}=Y$, and $\fig{1}{{F}^{-1}(v)}$ is a figure in $\mathscr{C}_1$.
  For every $a\in\fig{1}{{F}^{-1}(v)}$ there exists $y\in v$ which satisfies $a\doteq_1{F}^{-1}(y)$.
  By continuity of $\mathscr{F}$, $a$ satisfies ${F}(a)\doteq_2 y$, which means that ${F}(a)\in\mon{2}{y}\subseteq Y$, thus, $a\in F^{-1}(Y)$.
  Consequently, $\fig{1}{{F}^{-1}(v)}\ =\  F^{-1}(Y)$, thus, ${F}^{-1}(Y)$ is closed.

  To prove (ii) $\Rightarrow$ (iii), let $u\subseteq C_1$ be a connected set of $\mathscr{C}_1$.
  Then, for every $a\in u$, the inverse image of a monad of $y={F}(a)$, that is, ${F}^{-1}\left(\mon{2}{y}\right)$ is a closed figure of $\mathscr{C}_1$.
  Since $y={F}(a)$, ${F}^{-1}\left(\mon{2}{y}\right)$ contains $\mon{1}{a}$ as its subclass.
  It means that all the indiscernibles of $a\in u$  are included in the figure ${F}^{-1}\left(\mon{2}{y}\right)$.
  Thus, ${F}(\mon{1}{a})\subseteq\mon{2}{{F}(a)}$ follows.
  Since $u$ is connected, there exists $\theta\in N$ and a motion $\delta:\theta+1\rightarrow u$ which satisfies $\rng{(\delta)}=u$ and $\delta(\alpha)\doteq_1 \delta(\alpha+1)$ for all $\alpha\in\theta+1$.
  $\delta$ traces along indiscernibles, one by one, so does ${F}(\delta)$ which means that ${F}(\delta)$ is a trace of $\mathscr{C}_2$.
  Therefore, ${F}(u)$ is connected.
\end{proof}

When morphisms are defined on metric continua, continuity can be define by a standard $\varepsilon$-$\delta$ manner.
A morphism both from and to metric continua $\mathscr{F}:\mathscr{C}_1\rightarrow\mathscr{C}_2$ is continuous iff
\[
  \left(\forall e\in Q\right)
  \left(\exists d\in Q\right)
  \left(d>(d(x,a))\Rightarrow(e>d(\mathscr{F}(x),\mathscr{F}(a)))\right).
\]

Lastly, let us claim that continuous morphisms of AST are \textit{uniform} too, where a morphism $\mathscr{F}:\mathscr{C}_1\rightarrow\mathscr{C}_2$ is said to be \textit{uniformly continuous} iff
\[
(\forall k\in\FN)
(\exists j\in\FN)
(\forall x,y\in C_1)
((\langle x,y\rangle\in R_{1_j})\Rightarrow
(\langle F(x),F(y)\rangle \in R_{2_k}))
\]
where $(R_{1_j})_{j\in\FN}$ and $(R_{2_k})_{k\in\FN}$ are generating sequences of $\doteq_1$ and $\doteq_2$ respectively.

\begin{thm}
  Let $\mathscr{F}:\mathscr{C}_1\rightarrow \mathscr{C}_2$ be a continuous morphism.
  Then $\mathscr{F}$ is also uniformly continuous.
\end{thm}
\begin{proof}
  Suppose $\mathscr{F}$ is not uniformly continuous.
  Then there exists $k\in\FN$ which satisfies for all $j\in\FN$ there exists some $x,y\in C_1$,
  \[
  (\langle x,y\rangle\in R_{1_j}) \wedge (\langle F(x),F(y)\rangle\notin R_{2_k}).
  \]
  It implies that $x\doteq y$ but $F(x)\not\doteq F(y)$.
  It contradicts with $F$'s continuity.
\end{proof}

\section{Concluding Remarks}
The paper puts in order (1) a way to define basic concepts of topology in AST,
and shows (2) their correspondence with those of conventional ones, and
(3) isomorphicity  of a system of real numbers in AST to that of conventional one.

As it is widely known in model theory, there are many equivalent systems of real numbers other than $\mathbb{R}$ constructed within an axiomatic system of ZFC.
One of renowned examples may be that of nonstandard analysis, see Robinson \cite{rob66} or Davis \cite{Davis} for more detailed information.
Its way to construct a system is essentially the same as AST's,
that is, regarding monads as real numbers.

At odds with that popularity, the attempts regarding morphisms as real functions in the studies of nonstandard analysis are rarely seen except Tsujishita \cite{tjst}, as far as we know.
However, this line of investigations may seem fruitful, especially when one tries to apply them to capture the phenomena in the real world.

One benefit of that rests on its ability enabling us to see continuity of functions, or morphisms, from the more natural perspective.
Remind us here the $\varepsilon$-$\delta$ arguments, which regard continuity of functions indirectly as\textit{ having no gap} in its graph, since it says only that however small the gap is in the range, there always remains enough space left on the domain of the function, assigning the value picked from that domain to the function always resulted in the value staying within that range.
And the process confirming whether there remains unchecked gaps in the range left will not end and last forever by definition.

However, as we saw in the last section, by contrast, viewing continuity by way of motions enables us to check it directly.
The definition requires that morphisms $\mathscr{F}$ to preserve properties of motions, that is, for any motion $\delta$, composed of $\doteq$-chains with its trace, it must be satisfied that $\mathscr{F}(\delta)$ is also a motion.
It simply says that morphism is continuous since its graph is connected.

The key difference is that AST has a way to represent direct connection between two objects, enabled by indiscernibility equivalences $\doteq$, while $\varepsilon$-$\delta$ argument can only judge two are separated, enabled by open sets, which do not guarantee that they are connected in general.

This same ability of AST enables the authors to deal with the problem of e-mail game, that is, how agents can come to know that contents of e-mails are shared, in Sakahara and Sato \cite{cogjump}.
They come to know that they share the same information, not because there remains no possibility of not knowing, but they simply ignore indiscernible differences how many times they confirmed that.
They did it by, say, \textit{jump}.
It is the same jump which enable us to judge the motion of points has no gap in its trace.

It is also worth pointing out that real functions and their representations are defined on the same continua.
Consequently, results acquired from them reside in the same domain.
In contrast, when one works within a nonstandard framework, the system of reals and its extensions reside in different domains.
The result concerning, say, external elements, thus, have no room to reside in the standard domain.
They are simply nonexistent.

Nevertheless, it is not, of course, an essential problem of nonstandard analysis, since the arguments provided here can also be applied to nonstandard analysis.
It can also be said that the perspective this paper presents benefits nonstandard analysis too.
\bibliographystyle{plain}
\bibliography{ref}

\begin{thebibliography}{1}

\bibitem{Davis}
Martin Davis.
\newblock {\em Applied nonstandard analysis}.
\newblock Wiley New York, 1977.

\bibitem{rob66}
Abraham Robinson.
\newblock {\em Non-standard Analysis}.
\newblock North-Holland Publishing Co., Amsterdam, 1966.

\bibitem{cogjump}
Kiri {Sakahara} and Takashi {Sato}.
\newblock {An Alternative Set Model of Cognitive Jump}.
\newblock {\em arXiv e-prints}, arXiv:1904.00613, Apr 2019.

\bibitem{tjst}
Toru {Tsujishita}.
\newblock {Alternative Mathematics without Actual Infinity}.
\newblock {\em arXiv e-prints}, arXiv:1204.2193v2, Jun 2012.

\bibitem{ast}
Petr Vop\v{e}nka.
\newblock {\em Mathematics in the Alternative Set Theory}.
\newblock Teubner Verlagagesellshaft, Leipzig, 1979.

\bibitem{encycro-ast}
Petr Vop\v{e}nka and Kate\v{r}ina Trlifajov\'{a}.
\newblock Alternative set theory.
\newblock In Christodoulos~A. Floudas and Panos~M. Pardalos, editors, {\em
  Encyclopedia of Optimization}, pages 73--77. Springer, 2009.

\end{thebibliography}

\end{document}